\documentclass[12pt,reqno,oneside]{amsart}
\usepackage[asymmetric,top=3.5cm,bottom=4.0cm,left=3.1cm,right=3.1cm]{geometry}

\newtheorem{tm}{Theorem}[section]

\newtheorem{que}{Question}[section]

\theoremstyle{remark}
\newtheorem{rk}{Remark}[section]
\newtheorem{ex}{Example}[section] 

\numberwithin{equation}{section}

\usepackage{graphicx} 
\usepackage{extarrows}
\usepackage{latexsym}
\usepackage{amsmath}
\usepackage{amssymb}
\usepackage{amsfonts}
\usepackage{verbatim}
\usepackage{mathrsfs}
\usepackage[modulo]{lineno} % 行数
\usepackage[latin1]{inputenc}
\usepackage{color}
\usepackage{xcolor}
\usepackage[colorlinks,citecolor=blue,urlcolor=blue]{hyperref}
\usepackage{soul}

\newcommand{\ee}{\mathbb E}

\newcommand{\pp}{\mathbb P}
\newcommand{\nn}{\mathbb N}
\newcommand{\rr}{\mathbb R}

\newcommand{\AAA}{\mathcal A}
\newcommand{\BB}{\mathcal B}
\newcommand{\CC}{\mathcal C}
\newcommand{\DD}{\mathcal D}

\newcommand{\PP}{\mathcal P}

\newcommand{\FFF}{\mathscr F}

\newcommand{\<}{\langle}
\renewcommand{\>}{\rangle}
\allowdisplaybreaks \allowdisplaybreaks[4]

\newcommand{\dd}{\mathrm{d}}

\newcommand{\abs}[1]{\left\lvert #1 \right\rvert}
\newcommand{\norm}[1]{\left\lVert #1 \right\rVert}

\begin{document}

\title[Inversions of stochastic processes from ergodic measures of SDEs]
{Inversions of stochastic processes from ergodic measures of Nonlinear SDEs}

%    Information for the first author
\author{Hongyu LIU}
\address{Department of Mathematics, City University of Hong Kong, Kowloon, Hong Kong SAR, P.R. China}
% \curraddr{}
\email{hongyliu@cityu.edu.hk, hongyu.liuip@gmail.com}

%    Information for the second author
\author{Zhihui LIU}
\address{Department of Mathematics, Southern University of Science and Technology, Shenzhen 518055, P.R. China}
% \curraddr{}
\email{liuzh3@sustech.edu.cn, liuzhihui@lsec.cc.ac.cn}
\thanks{The work of H. Liu is supported by the Hong Kong RGC General Research Funds (projects 11311122, 11300821, and 11303125), the NSFC/RGC Joint Research Fund (project  N\_CityU101/21), the France-Hong Kong ANR/RGC Joint Research Grant, A-CityU203/19. The work of Z. Liu is supported by the National Natural Science Foundation of China, No. 12101296, Guangdong Basic and Applied Basic Research Foundation, No. 2024A1515012348, and Shenzhen Basic Research Special Project (Natural Science Foundation) Basic Research (General Project), Nos. JCYJ20220530112814033 and JCYJ20240813094919026.} 

%    General info

 

\begin{abstract}
We introduce and analyze a novel class of inverse problems for stochastic dynamics: Given the ergodic invariant measure of a stochastic process governed by a nonlinear stochastic ordinary or partial differential equation (SODE or SPDE), we investigate the unique identifiability of the underlying process--specifically, the recovery of its drift and diffusion terms. This stands in contrast to the classical problem of statistical inference from trajectory data. 
 We establish unique identifiability results under several key scenarios, including cases with both multiplicative and additive noise, for both finite- and infinite-dimensional systems. Our analysis leverages the intrinsic structure of the governing equations and their quantitative relationship with the ergodic measure, thereby transforming the identifiability problem into a uniqueness issue for the solutions to the associated stationary Fokker-Planck equations. This approach reveals fundamental differences between drift and diffusion inversion problems and provides counterexamples where unique recovery fails.
This work lays the theoretical foundation for a new research direction with significant potential for practical application. 
\end{abstract}

\maketitle

\section{Introduction} 

Stochastic differential equations (SDEs), both SODEs and SPDEs, serve as fundamental mathematical models for phenomena across the sciences and engineering, from particle dynamics in physics to interest rate fluctuations in finance and pattern formation in biology. A profound feature of many such systems, particularly those with dissipative nonlinearities, is the existence and uniqueness of an \emph{ergodic (invariant) measure}. This measure characterizes the long-time, statistical equilibrium of the system, describing the distribution of the state process after transients have decayed. In essence, while individual trajectories remain unpredictable, the probability law of the process converges to this invariant measure, offering a powerful statistical description of the system's steady state.

The \emph{forward problem}---understanding the properties of the invariant measure given the governing equations (i.e., the drift and diffusion coefficients)---has been a central theme in stochastic analysis for decades. A vast body of literature is devoted to proving the existence, uniqueness, and regularity of invariant measures, as well as quantifying the rate of convergence to equilibrium; see e.g. \cite{BS20, DZ96, HM06, Liu25, Liu26, LL25, MSH02, Wan23} and the references cited therein. In parallel, the field of statistical inference for SDEs is highly developed, where the central task is to estimate the unknown parameters of the drift and diffusion terms from \emph{finite-time, noisy observational data} of the process trajectories; see e.g. \cite{DD22, LLM21, LW24, LLv25, LM23, LSW25, LW25, WH20} and the references cited therein. 

In this paper, we address a fundamentally different and, to our knowledge, largely unexplored \emph{inverse problem}. We posit a scenario where the \emph{ergodic measure} itself is known, or can be reliably estimated, and we ask: {Can one uniquely recover the underlying stochastic process---namely, its drift and diffusion terms---from this measure alone?} We term this the \emph{inversion of a stochastic process from its ergodic measure}.

This problem is motivated by several compelling considerations:
\begin{enumerate}
    \item \textbf{A New Data Paradigm:} In many complex systems, especially those at a macroscopic or thermodynamic scale, one may have access to statistical equilibrium data (e.g., from long-time averages or ensemble measurements) without direct access to the complete dynamical trajectories. Our formulation directly utilizes this equilibrium information as the primary data for inversion.
    \item \textbf{Theoretical Foundation:} Establishing identifiability---whether the map from the coefficients to the invariant measure is injective---is a prerequisite for any well-posed inverse problem. If two different drift-diffusion pairs can generate the same invariant measure, then any recovery attempt is inherently futile. Our work provides a rigorous foundation for this new class of problems.
    \item \textbf{Potential for Robustness:} Ergodic measures are global, time-averaged objects. An inversion methodology based on them could potentially be more robust to model misspecification and high-frequency noise than methods relying on precise trajectory data.
\end{enumerate}

Despite its natural motivation, this inverse problem presents significant challenges. The relationship between the coefficients of a nonlinear SDE and its invariant measure is highly implicit, governed by the so-called stationary Fokker-Planck equation. These are second-order elliptic equations in high or infinite dimensions, and the inversion is inherently ill-posed without suitable \emph{a priori} assumptions.

The main contributions of this paper are the establishment of {unique identifiability results} for this novel inverse problem. We derive sufficient conditions under which the knowledge of the ergodic measure guarantees the unique recovery of the unknown drift and diffusion terms. Our analysis covers several key scenarios: 
\begin{enumerate}
\item
Drift inversion: Unique recovery is possible for one-dimensional SODEs with multiplicative noise (see Theorem \ref{tm-b}), high-dimensional Langevin equations (i.e., with gradient drifts and additive noise in Theorem \ref{tm-lan-b}), and certain SPDEs with additive noise (see Theorem \ref{tm-spde}), but generally fails for high-dimensional SODEs without gradient structure, even in the additive noise case (we construct a counterexample in Theorem \ref{tm-b-i}).
 
\item
Diffusion inversion: Unique recovery is possible for any dimensional SODEs or certain SPDEs driven by additive noise (see Theorems \ref{tm-s}, \ref{tm-lan-s}, \ref{tm-s+}, \ref{tm-spde}), but generally fails in the multiplicative noise case, even for one-dimensional SODEs (we construct a counterexample in Theorem \ref{tm-s-i}), let alone SPDEs. 
\end{enumerate}

Our technical approach is rooted in a detailed analysis of the stationary Fokker-Planck equations. We transform the problem of identifiability into the problem of uniqueness of solutions to certain PDEs. The proofs rely on a careful exploitation of the quantitative relationships between the coefficients and the density of the invariant measure.

To the best of our knowledge, this work represents a new direction in research on inverse problems for stochastic dynamics. The established identifiability results provide the theoretical bedrock upon which future computational and statistical recovery algorithms can be built. We envision many potential developments of practical significance, including applications to model validation, uncertainty quantification for systems in equilibrium, and the design of stochastic processes with desired statistical properties.

The rest of this paper is organized as follows. Some preliminaries are given in Section \ref{sec2}.
In Section \ref{sec3}, we investigate the drift and diffusion inversion problem for nonlinear SODEs.
These ideas are generalized in Section \ref{sec4} to the infinite-dimensional case.

\section{Preliminaries}
 \label{sec2}

This section will investigate the drift and diffusion inversion problems for nonlinear SODEs driven by an $\rr^m$-valued Wiener process $W$ on a complete filtered probability space $(\Omega, \FFF, \mathbb{F}, \pp)$ (with $\mathbb{F}:=(\FFF(t))_{t\ge 0}$):
\begin{align}\label{sode}
    \dd X(t) = b(X(t))\dd t + \sigma(X(t)) \dd W(t), ~ t \ge 0, 
\end{align}
where $b: \rr^d \to \rr^d$ and $\sigma: \rr^d \to \rr^{d \times m}$ belongs to the space $\BB_b^\text{loc}$ of Borel measurable and locally bounded functions. 
Throughout, we denote by $\abs{\cdot}$ and $\<\cdot, \cdot\>$ the Euclidean norm and inner product in $\rr^d$ or the Hilbert space in Section \ref{sec4} if there is no confusion and by $\norm{\cdot}$ the Hilbert--Schmidt norm in $\rr^{d \times m}$.

\subsection{Fokker--Planck equation and Invariant measure}

Let us recall the invariant measure of Eq. \eqref{sode} and the connection to the Fokker--Planck equation.

Assume that Eq. \eqref{sode} has a unique strong solution which is a homogeneous Markov process (under broad assumptions on the coefficients of Eq. \eqref{sode}; see Example \ref{ex1}).
Denote by $P_t$ the associated Markov semigroup of Eq. \eqref{sode} on $\rr^d$, i.e.,
\begin{align} \label{df-pt}
P_tf(x):= \ee^x [f(X_t)], \quad t\ge 0, ~ f\in \BB_b(\rr^d), ~ x\in \rr^d,
\end{align}
where $\BB_b(\rr^d)$ is the class of all bounded measurable functions on $\rr^d$ and $\ee^x$ is the expectation for the solution $X_t$ of \eqref{spde} with $X_0=x$.  
A probability measure $\mu$ on $\BB(\rr^d)$ is called \emph{invariant} for (the solution $\{X_t\}_{t \ge 0}$) Eq. \eqref{sode} if 
\begin{align} \label{df-im}
\int_{\rr^d} P_t \phi(x) \mu({\rm d}x)=\mu(\phi):=\int_{\rr^d} \phi(x) \mu({\rm d}x),
\quad \forall~ \phi \in \BB_b(\rr^d), ~ t \ge 0.
\end{align} 
This is equivalent to $\mu(A)=\int_{\rr^d} P_t(x, A) \mu({\rm d}x)$ for all $A \in \BB(\rr^d)$,
where $\{P_t(x, \cdot)\}_{t \ge 0}$ denote the transition probabilities of $\{X_t\}_{t \ge 0}$ starting from $X_0=x \in \rr^d$; $\mu$ is also called a stationary distribution of $\{X_t\}_{t \ge 0}$ or Eq. \eqref{sode}.
An invariant (probability) measure $\mu$ of Eq. \eqref{sode} is called \emph{ergodic} if 
\begin{align} \label{df-erg}
\lim_{T \to \infty} \frac1T \int_0^\infty P_t \phi(x)=\mu(\phi) 
\quad \text{in}~ L^2(\rr^d; \mu),\quad \forall~ \phi \in L^2(\rr^d; \mu).
\end{align}
It is well-known that if $\{X_t\}_{t \ge 0}$ admits a unique invariant measure, then it is ergodic; in this case, we call it uniquely ergodic.

As $b^i$ and $D^{ij}$, $i,j=1,\cdots,d$, are locally bounded, the transition probabilities $\mu_t$ of Eq. \eqref{sode} at the time $t>0$ satisfy the following parabolic Fokker--Planck equation (Fokker-Planck equation) for measures (in the sense of generalized functions; see, e.g., \cite[Proposition 1.3.1]{BKRS15}):
\begin{align*}
\frac{\partial}{\partial_t} \mu_t
=L_{D,b} \mu_t
:= \partial_{x_i} \partial_{x_j} [D^{ij} \mu_t] - \partial_{x_i} [b^i \mu_t], 
\quad t>0.
\end{align*} 

Moreover, the invariant measure $\mu$ defined by \eqref{df-im} satisfies the following stationary Fokker-Planck equation for measures:
\begin{align*}
L_{D,b} \mu= \partial_{x_i} \partial_{x_j} [D^{ij} \mu] - \partial_{x_i} [b^i \mu]=0, 
\quad t>0.
\end{align*}  
Let $\mu$ have a positive density $p$.
Then the corresponding stationary Fokker-Planck equation for densities becomes % p169
\begin{align} \label{eq-Fokker-Planck}
\partial_{x_i} \partial_{x_j} [D^{ij} p]- \partial_{x_i} [b^i p]=0.  
\end{align}

\subsection{Measurement operator}

Let $D:=\sigma \sigma^\top$ be the diffusion tensor Eq. \eqref{sode} with elements
\begin{align} \label{df-D}
D^{ij}=\frac12 \sigma^{il} \sigma^{jl},
\quad i,j=1,2,\cdots,d,
\end{align}
where we always assume that the summation is taken over all repeated indices.
From Eq. \eqref{eq-Fokker-Planck} we observe that a pair of $(b, D)$ (even with different $\sigma: \rr^d \to \rr^{d \times m}$ generating $D$ through \eqref{df-D}) satisfies the same stationary Fokker-Planck equation \eqref{eq-Fokker-Planck} for densities.
 
Denote by $\AAA$ the set of assumptions on $(b, D) \in \BB_b^\text{loc}(\rr^d; \rr^d) \times \BB_b^\text{loc}(\rr^d; \rr^{d \times d})$ with $D=\sigma \sigma^\top/2$ being nondegenerate such that Eq. \eqref{sode} has a unique strong solution with a unique ergodic measure $\pi_{b, D} \in \PP(\rr^d):=\{\text{probabily measures on}~\rr^d\}$ having a positive density $p_{b, D} \in \DD(\rr^d):=\{\text{density functions on}~\rr^d\}$.
We note that $p_{b, D}$ is $\lambda${\rm -a.e.} uniquely determined by $\pi_{b, D}$, where $\lambda$ denotes the Lebesgue measure.
It is clear $\AAA$ is nonempty as this assumption is valid under broad conditions that include all the examples given in Sections \ref{sec2} and \ref{sec3}; see, e.g., \cite[Theorems 1.5.2, 3.1.2, and 6.3.1]{BKRS15}. 
Then for each pair of $(b, D) \in \AAA$, we obtain a solution $\pi_{b, D} \in \PP(\rr^d)$ with density $p_{b, D} \in \DD(\rr^d)\}$ and we can define the measurement operator 
\begin{align}
\begin{split}
T_{b, D}: \AAA & \longrightarrow p_{b, D}(\AAA) \subset \DD(\rr^d) \\
(b, D) &\longmapsto p_{b, D}=T_{b, D}(b, D).
\end{split}
\end{align}
Our main aim is to find the inversion of $b$ and/or $D$ from the ergodic measure $\pi_{b, D}$ of Eq. \eqref{sode}.

Later in Section \ref{sec2} (see Remark \ref{rk-bs}), we will observe that given an ergodic measure $\pi_{b, D}$, generally one cannot simultaneously invert the drift $b$ and diffusion tensor $D$.  
So we also consider the following measurable operators with a fixed drift or diffusion.
For a fixed $D \in \BB_b^\text{loc}(\rr^d; \rr^{d \times d})$, define
\begin{align}  \label{t-s}
\begin{split}
T_b: \AAA_b =\{b:~(b, D) \in \AAA\} & \longrightarrow T_b(\AAA_b) \subset \DD(\rr^d) \\
b & \longmapsto p_{b, D}=T_b(b, D),
\end{split}
\end{align} 
and for a fixed $b \in \BB_b^\text{loc}(\rr^d; \rr^d)$, define
\begin{align}  \label{t-b}
\begin{split}
T_D: \AAA_D =\{D:~(b, D) \in \AAA\} & \longrightarrow T_D(\AAA_D) \subset \DD(\rr^d) \\
D & \longmapsto p_{b, D}=T_D(b, D).
\end{split}
\end{align} 
It is clear that $T_b$ and $T_D$ are well-defined and highly nonlinear.

By the definitions of $\AAA_b$ and $\AAA_D$ in \eqref{t-s} and \eqref{t-b}, respectively, it is clear that $T_b: \AAA_b \to T_b(\AAA_b)$ and $T_D: \AAA_D \to T_D(\AAA_D)$ are surjective.
Now a natural and central question is whether $T_b: \AAA_b \to T_b(\AAA_b)$ and $T_D: \AAA_D \to T_D(\AAA_D)$ are injective?  
If they are injective (and thus bijective), are their inverses, denoted by $T_b^{-1}$ and $T_D^{-1}$, continuous?

\begin{que}
\begin{enumerate}
\item
For a fixed $\sigma$ such that $(b_k, D) \in \AAA$ with $k=1,2$, whether $p_{b_1, D}=p_{b_2, D}$ induces $b_1=b_2$ (equivalently, $T_b^{-1}$ exists)?
If so, is $T_b^{-1}$ continuous in the sense that $\lim_{p_{b_1, D} \to p_{b_2, D}} T_b^{-1}(p_{b_1, D})=T_b^{-1}(p_{b_2, D})$?

\item
For a fixed $b$ such that $(b, D_k) \in \AAA$ with $k=1,2$, whether $p_{b, D_1}=p_{b, D_2}$ if and only if $D_1=D_2$ (equivalently, $T_D^{-1}$ exists)?
If so, is $T_D^{-1}$ continuous in the sense that $\lim_{p_{b, D_1} \to p_{b, D_2}} T_D^{-1}(p_{b, D_1})=T_D^{-1}(p_{b, D_2})$?
\end{enumerate} 
\end{que}

\section{Drift and Diffusion Inversion Problem for SDEs}
 \label{sec3}

In this section, we study the drift inversion problem and the diffusion inversion problem with both multiplicative and additive noise, for both one-dimensional and high-dimensional SODEs.
Meanwhile, we give several examples to indicate the ill-posedness of both the drift inversion problem for high-dimensional SODEs, even driven by additive noise, and the diffusion problem for SODEs driven by multiplicative noise, even in one dimension.

\subsection{Drift inversion for SODEs}

We have the following bijectivity of $T_b$ in the one-dimensional case (i.e., $d=1$) driven by $1 \le m$-dimensional Wiener process.
In this case, the nondegenerate condition of $D$ reduces to 
$|\sigma(x)|>0$ for any $x \in \rr$.

\begin{tm}  \label{tm-b}
For a fixed $D=D_k \in \CC(\rr; \rr)$, assume that $(b_k, D_k) \in \AAA$ and 
\begin{align} \label{con-1}
\int_\rr \frac{e^{U_k(x)}}{D_k(x)} {\rm d}x < \infty,
\end{align}
where $U_k \in \CC^1(\rr; \rr)$ is any premitive of $b_k/D$, i.e., 
$U_k'=b_k /D$, $k=1,2$. 
Then $\pi_{b_1, D}=\pi_{b_2, D}$ if and only if $b_1=b_2$, i.e., $T_b$ is a bijective.
\end{tm}

\begin{proof}   
In the case of $d=1$, the stationary Fokker-Planck equation \eqref{eq-Fokker-Planck} reduces to the following linear ODE with variable coefficients:
\begin{align}   \label{eq-den}
(D p)''- (b p)' =0.
\end{align}   
Solving the above equation \eqref{eq-den}, we obtain its general solution
\begin{align*}
p(x) & =\frac{C}{2 D(x)} \int_{x_0}^x
 \exp\Big(\int_y^x \frac{b}{D} {\rm d}z \Big) {\rm d}y 
 + \frac{D(x_0) p(x_0)}{D(x)} \exp\Big(\int_{x_0}^x \frac{b}{D}  {\rm d}z \Big), 
~ x \in \rr,
 \end{align*}
for any $x_0 \in \rr$ and any constant $C \in \rr$.
Here and after, we omit the integral variable to lighten the notation.
 The requirement $\int_{-\infty}^\infty p(x) {\rm d}x=1$ implies that $C=0$.
Consequently, $\pi_{b, D}(dx)=p_{b, D}(x) dx$ with  
 \begin{align} \label{den}
p_{b, D}(x)=\frac{C}{D(x)} \exp\Big(\int_{x_0}^x \frac{b}{D}  {\rm d}y \Big), \quad x \in \rr,
 \end{align}
for any $x_0 \in \rr$ and $C=D(x_0) p(x_0)>0$ (satisfying $\int_{-\infty}^\infty p(x) {\rm d}x=1$) is the unique invariant measures of Eq. \eqref{sode}, provided \eqref{con-1} holds.

Let $U$ be any primitive of $b/D$.
Then the solution of Eq. \eqref{eq-Fokker-Planck} is    
 \begin{align*}
p_{b, D}(x)=C e^{-U(x_0)}\frac{e^{U(x)}}{D(x)}, \quad x \in \rr,
 \end{align*}
for any $x_0 \in \rr$ and $C = D(x_0) p(x_0)=e^{U(x_0)}(\int_\rr e^U/D {\rm d}x )^{-1}>0$ (satisfying $\int_{-\infty}^\infty p(x) {\rm d}x=1$).
We can rewrite it as 
\begin{align} \label{inv}
p_{b, D}(x)=\frac{e^{U(x)}}{D(x)} \Big(\int_\rr \frac{e^{U}}{D} {\rm d}y \Big)^{-1}, \quad x \in \rr,
 \end{align}
 
 Now let $D_1=D_2=D$ and $\pi_{b_1, D}=\pi_{b_2, D}$, we have the existence of a constant $C$ such that  
\begin{align*}
\frac{e^{U_1(x)}}{D(x)} =C \cdot \frac{e^{U_2(x)}}{D(x)},  
 \end{align*}
 from which we obtain 
 \begin{align*}
U_1(x)-U_2(x)=C.
 \end{align*}
 Differentiating on both sides, we obtain $ b_1 = b_2$, and thereby prove the injectivity of $T_b$.
\end{proof}

\begin{rk} \label{rk-bs}
From the representation \eqref{inv}, even in the one-dimensional case, there exist many pairs $(b_k, D_k) \in \AAA$ that generalize the same ergodic measure, so that one can not invert both the drift and diffusion from a given ergodic measure simultaneously.
\end{rk}

The following result shows that, different from the one-dimensional case, $T_b$ is generally not injective for the high-dimensional SODEs, even in the additive noise case.

\begin{tm}  \label{tm-b-i}
For a fixed $D=D_k \in \CC(\rr^d; \rr^{d \times d})$ with $d \ge 2$, assume that $(b_k, D_k) \in \AAA$ satisfy \eqref{con-1} with $U_k'=b_k /D$, $k=1,2$.  
Then $\pi_{b_1, D}=\pi_{b_2, D}$ if and only if there exists a continuously differentiable function $\psi: \rr^d \to \rr^d$ such that 
$(b_1-b_2) p=\nabla \times \psi$.
In particular, there exist $b_1 \neq b_2$, together with a fixed $D \in \rr^{d \times d}$, such that $\pi_{b_1, D}=\pi_{b_2, D}$, thus $T_b$ is not injective. 
\end{tm}

\begin{proof} 
It follows from the stationary Fokker-Planck equation \eqref{eq-Fokker-Planck} for densities with $D_1=D_2=D$ that $\pi_{b_1, D}=\pi_{b_2, D}$ if and only if 
\begin{align}  \label{eq-b}  
\nabla \cdot [(b_1-b_2) p]=0.
\end{align} 
When $d \ge 2$, by Helmholtz decomposition theorem (i.e., the space of vector fields can be decomposed into the range of the gradient and the kernel of the divergence), there exists a continuously differentiable function $\psi: \rr^d \to \rr^d$ such that 
\begin{align*}
(b_1-b_2) p=\nabla \times \psi.
\end{align*}  

Finally, we construct different $b_k$' so that $\pi_{b_1, D}=\pi_{b_2, D}$.
Let $\pi_{b_1, D}(x)=p(dx)$ be the unique ergodic measure of Eq. \eqref{sode} with coefficients $(b_1, D) \in \AAA$. 
We take a skew-smmetric constant matrix $J=(J^{ij}) \in \rr^{d \times d}$ (i.e., $J=-J^\top$) and 
\begin{align*}
b_2:=b_1- J \nabla \ln p.
\end{align*}
Then, direct calculations yield that 
\begin{align*}
\nabla \cdot (b_1-b_2) p=\nabla \cdot [J (\nabla \ln p) p]
=\nabla \cdot [J \nabla p]=J^{ij} \partial_{x_i} \partial_{x_j} p
=0,
\end{align*}  
by the skew-symmetry of $J$, provided $p$ is twice-differentiable, which is valid for a broad range of assumptions.

To ensure the (well-posedness and) unique ergodicity of the above counterexample, for example, we consider the Langevin equation \eqref{eq-lan} with density \eqref{lan-d} and $U \in \CC^2(\rr^d, \rr)$.
Suppose that $b_1(x)=\nabla U(x)=-\alpha x$ with $\alpha>0$ and $\sigma=\sqrt{\beta} {\rm Id}_d$.
It is clear that $(b_1, \sigma_1)$ satisfy all the conditions of Example \ref{ex1}.
By elementary calculations, we obtain that $b_2:=b_1-J \nabla \ln p=b_1-J b_1$ and the same $\sigma=\sqrt{\beta} {\rm Id}_d$ (so that the nondegenerate condition \eqref{non} holds) satisfy all the conditions of Example \ref{ex1}, with possibly different constants $L_i$, $i=0,1,\cdots,4$.
Indeed, as $\sigma$ is constant and $J$ is a skew-smmetric constant matrix implying $\<Jx, x\>=0$ for any $x \in \rr^d$, we have 
\begin{align*} 
& 2 \<b_2(x)-b_2(y), x-y\>+ \|\sigma(x)-\sigma(y)\|^2 \\
& =2 \<b_1(x)-b_1(y), x-y\> - 2 \<J (b_1(x)-b_1(y)), x-y\> \\
& = 2 \<b_1(x)-b_1(y), x-y\>
\le L_0|x-y|^2,  
\end{align*}   
which shows \eqref{mon-} with the same constant $L_0$.
A similar idea yields \eqref{coe-} and \eqref{pol-} with the same constants $L_1, L_2$ and different $L_4$ with $q=1$ (and $L_3=0$).  

This indicates that $\pi_{b_1, D}(x)= p(dx)$ is also the unique ergodic measure of Eq. \eqref{sode} with coefficients $(b_2, D)$, thus $T_b$ is not injective. 
\end{proof}

From Theorem \ref{tm-b-i}, it is known that $T_b$ is generally not injective in the high-dimensional, multiplicative noise case; the following result shows that they are injective for the following $d$-dimensional Langevin equation driven by $m=d$-dimensional Wiener process $W$:
\begin{align}\label{eq-lan}
    \dd X(t) =\nabla U(X(t)) \dd t + \sqrt{\beta} \dd W(t), ~ t \ge 0,  
\end{align}
where $\beta>0$ is a constant and $U: \rr^d \to \rr$ is Borel measurable. 
This corresponds to Eq. \eqref{sode} with $b=\nabla U$ and $\sigma=\sqrt \beta {\rm Id}_d$, where ${\rm Id}_d$ is the $d$-dimensional unit matrix, so that $D$ given in \eqref{df-D} is the constant $\beta/2$ (or $D=\frac12 \beta {\rm Id}_d$). Thus, the stationary Fokker-Planck equation \eqref{eq-Fokker-Planck} for densities reduces to
\begin{align} \label{eq-Fokker-Planck-d}
\frac\beta 2 \Delta p - \nabla \cdot [b p]=0.  
\end{align}  

From now on, all the invariant measure $\pi_{b, D}$, invariant density $p_{b, D}$, and measurement operator $T_{b, D}$, are denoted by $\pi_{b, \beta}$, $p_{b, \beta}$, and $T_{b, \beta}$, respectively, in the Langevin equation case.
By direct calculations, the following invariant density, called the Gibbs density, satisfies Eq. \eqref{eq-Fokker-Planck-d} (see also in, e.g., \cite[Theorem 4.1.11]{BKRS15}):
\begin{align}\label{lan-d}
p_{b, \beta}(x)=e^{\frac{2U(x)}{\beta}} \Big(\int_{\rr^d} e^{\frac{2U(x)}{\beta}}{\rm d}x \Big)^{-1}, 
\quad x \in \rr^d,
\end{align} 
provided the following $d$-dimensional similar integrability condition as \eqref{con-1} holds:
\begin{align}\label{con-d}
\int_{\rr^d} e^{\frac{2U(x)}{\beta}}{\rm d}x<\infty.
\end{align}
Therefore, $\pi_{b, \beta}(\dd x)= p_{b, \beta}(x) \dd x$, with $p_{b, \beta}$ given above, is the unique ergodic measure of the $d$-dimensional Langevin equation \eqref{lan-d}, provided \eqref{con-d} holds.

\begin{tm} \label{tm-lan-b}
For a fixed $\beta>0$, assume that $(b, \beta) \in \AAA$ satisfying \eqref{con-d} for any premitive $U \in \CC^1(\rr^d; \rr)$ of $b$. 
Then $T_b$ is bijective and its inverse $T_b^{-1}$ satisfy
\begin{align} 
T_b^{-1}(p_{b, \beta})
(=b) & =\frac\beta2 \nabla \ln p_{b, \beta}, \label{b-p}  \\
|T_b^{-1}(p_{b_1, \beta})-T_b^{-1}(p_{b_2, \beta})| 
(=|b_1-b_2|)
& =\frac\beta2  |\nabla \ln p_{b_1, \beta}-\nabla \ln p_{b_2, \beta}|.  \label{sta-b}
\end{align}    
\end{tm}

\begin{proof} 
It follows from the representation \eqref{lan-d} that 
\begin{align*}
\ln p(x)= \frac{2U(x)}{\beta} +C, 
\quad x \in \rr^d,
\end{align*} 
from which we derive \eqref{b-p} and \eqref{sta-b}. 
\end{proof}

\subsection{Diffusion inversion for SODEs}

For the one-dimensional diffusion inversion problem, from the representation \eqref{inv} of the density, i.e., the explicit solution of Eq. \eqref{eq-Fokker-Planck}, we derive a gauge equivalence condition for $\sigma$ such that $T_D$ is injective.
This particularly indicates that $T_D$ is bijective in the case of additive noise. 
In contrast, it is not surjective, in the usual sense, for the multiplicative noise case.

\begin{tm} \label{tm-s} 
For a fixed $b=b_k \in \CC(\rr; \rr)$, assume that $(b_k, D_k) \in \AAA$ satisfing \eqref{con-1} for any premitive $U_k \in \CC^1(\rr; \rr)$ of $b/D_k$, $k=1,2$.
Then $\pi_{b, D_1}=\pi_{b, D_2}$ if and only if
\begin{align} \label{inv-s}
\frac{D_1(x)}{D_2(x)}
=1 + \frac{D_1(x_0)-D_2(x_0)}{D_2(x_0) \exp(\int_{x_0}^x b/D_2 {\rm d}y)},
\quad \forall ~ x, x_0 \in \rr.
 \end{align}   
 In particular, for the additive noise case, $\pi_{b, D_1}=\pi_{b, D_2}$ if and only if $D_1= D_2$, i.e., $T_D$ is a bijective. 
\end{tm}

\begin{proof} 
Let $\pi_{b, D_1}=\pi_{b, D_2}$.
By the representation \eqref{den} with any $x_0 \in \rr$, we have the existence of a positive constant $C=C(x_0)$ s.t.
\begin{align} \label{mul-den-s}
D_1(x)^{-1} \exp\Big(\int_{x_0}^x  \frac{b}{D_1}{\rm d}y\Big) 
=C D_2(x)^{-1} \exp\Big(\int_{x_0}^x \frac{b}{D_2} {\rm d}y\Big). 
 \end{align}
 Denote by $f(x):=C D_2(x)^{-1} \exp(\int_{x_0}^x b/D_2 {\rm d}y)$.
 It is equivalent to finding $D_1$ from 
 \begin{align*}
D_1(x)^{-1} \exp\Big(\int_{x_0}^x  \frac{b}{D_1}{\rm d}y\Big)
=f(x), \quad \lambda{\rm -a.e.} 
 \end{align*}
 This equation can be rewritten as the following one-order linear ODE with variable coefficients by setting $g:=D_1$:
 \begin{align} \label{ode-s}
g'(x) + \frac{f'(x)}{f(x)} g(x) = b(x).
 \end{align} 
 
By \eqref{mul-den-s}, it is known that the above linear equation, with 
$$\frac{f'(x)}{f(x)}= \frac{b-D_2'}{D_2},$$ 
has a solution given by $ g_1 = D_2$.
Moreover, its general solution is 
\begin{align*}
g(x) & =\int_{x_0}^x \exp\Big(\int_y^x -\frac{f'}{f} {\rm d}z\Big) b(y) {\rm d}y 
+ g(x_0) \exp\Big(\int_{x_0}^x -\frac{f'}{f} {\rm d}z\Big) \\
& = \frac1{f(x)} \int_{x_0}^x f b {\rm d}y 
+ g(x_0) \frac{f(x_0)}{f(x)} \\ 
& =D_2(x) \Big(1 + \frac{g(x_0)-D_2(x_0)}{D_2(x_0) \exp(\int_{x_0}^x b/D_2 {\rm d}y)}\Big),
 \end{align*}
 for any $x_0 \in \rr$, i.e., \eqref{inv-s} holds.
\iffalse
\begin{align} \label{}
\exp \Big(\int_{x_0}^x \frac{b}{D_2} {\rm d}y \Big) \Big(\frac{D_1(x)}{D_2(x)}-1\Big) = \frac{D_1(x)}{D_2(x)}-1,
\quad \forall ~ x, x_0 \in \rr.
 \end{align}
 \fi
 This shows that $D_1 \equiv D_2$ if and only if $D_1(x_0)=D_2(x_0)$ for some point $x_0 \in \rr$.
 
In the additive noise case, suppose that $D_1 \neq D_2$.
From \eqref{inv-s} we have 
\begin{align*}
\exp\Big(\frac1{D_2} \int_{x_0}^x b {\rm d}y\Big)
\equiv \frac{D_1-D_2}{(D_1/D_2-1) D_2} = 1,
\quad \forall ~ x_0 \in \rr.
 \end{align*}   
Then we derive $b \equiv 0$, but this is impossible since, in this case, Eq. \eqref{sode} is not uniquely ergodic. 
Therefore, $D_1 = D_2$, and thereby prove the bijectivity of $T_D$. 
\end{proof}

The following result shows that, different from the additive noise case, $T_D$ is generally not injective for SODEs driven by multiplicative noise, even in one dimension.

\begin{tm} \label{tm-s-i}  
There exist $D_1 \neq D_2$ which are not constant simultaneously, i.e., there do not exist constant vectors $(C_1, C_2) \in \rr^2$ such that $(D_1(x), D_2(x))=(C_1, C_2)$ for all $x \in \rr$, together with a fixed $b \in \CC(\rr; \rr)$, such that $\pi_{b, D_1}=\pi_{b, D_2}$, thus $T_D$ is not injective. 
\end{tm}

\begin{proof} 
Let $d=1$ and $(b, D_2) \in \AAA$ satisfy \eqref{con-1} for some $U_2 \in \CC^1(\rr; \rr)$ with $U_2=b/D_2$. 
Based on \eqref{inv-s}, for the multiplicative noise case, if we choose, e.g., $x_0=0$ and $D_1(0)=D_2(0)+1$, one can take another solution of Eq. \eqref{mul-den-s} as 
 \begin{align} \label{inv-1}
D_1(x)
% & =D_2(x) \Big[1 + |D_2(0)|^{-2} \exp\Big(-\int_{x_0}^x \frac{b}{D_2} {\rm d}y\Big)\Big]
& =D_2(x) (1 + C e^{-U_2(x)}), \quad x \in \rr,
 \end{align}
 with $C=[D_1(0)/D_2(0)-1] e^{U_2(0)}$.
 As $D_2(0) > 0$, it is clear that 
 $D_1(x) > D_2(x)$ so that $|D_1(x)| \neq |D_2(x)|$ for any $x \in \rr$.
  
It remains to ensure the unique ergodicity of the above counterexample; for example, we choose 
\begin{align*}
b(x)=-\frac{2x}{1+x^2}, \quad \sigma_2(x)=\sqrt{2}, \quad D_2(x)=1, 
\quad U_2(x)=- \ln (1+x^2), \quad x \in \rr.
\end{align*}
Then by \eqref{inv-1} we get 
\begin{align*}
D_1(x) =2+x^2, \quad 
\sigma_1(x) =\sqrt{2(2+x^2)}, \quad 
U_1(x)=\ln \frac{2+x^2}{1+x^2} , \quad x \in \rr.
\end{align*}
Elementary calculations show 
\begin{align*}
\sup_{x \in \rr} b'(x)=\sup_{x \in \rr} \frac{-2(1+2 x^2)}{\sqrt{1+x^2}}=-2, \quad x \in \rr, \\
\sup_{x \in \rr}  |\sigma_1'(x)|=\frac{\sqrt 2 |x|}{\sqrt{2+x^2}}=\sqrt 2, \quad 
\sup_{x \in \rr}  |\sigma_2'(x)|=0, \quad x \in \rr,
\end{align*}
thereby indicate that $(b, \sigma_1)$ and $(b, \sigma_1)$ satisfy all the conditions in Example \ref{ex1}.

Moreover,  one has
\begin{align*}
\int_\rr \frac{e^{U_2(x)}}{D_2(x)} {\rm d}x=
\int_\rr \frac{e^{U_1(x)}}{D_1(x)} {\rm d}x=\int_\rr \frac1{1+x^2} {\rm d}x =\pi<\infty.
\end{align*}
thereby shows the integrability condition \eqref{con-1} for these two pairs. 
\end{proof}

From Theorem \ref{tm-s-i}, it is known that $T_D$ (denoted by $T_\beta$ from now on) is generally not injective in the multiplicative noise case; the following result shows that they are injective for the $d$-dimensional Langevin equation \eqref{eq-lan}.

\begin{tm} \label{tm-lan-s}
For a fixed $\beta>0$, assume that $(b_k, \beta) \in \AAA$ satisfying \eqref{con-d} for any premitive $U_k \in \CC^1(\rr^d; \rr)$ of $b_k$, $k=1,2$.
Then $T_\beta$ is bijective and its inverse $T_\beta^{-1}$ satisfies 
\begin{align} 
T_\beta^{-1}(p_{b, \beta})
(=\beta) & \equiv \frac{2 b^i}{\partial_{x_i} \ln p_{b, \beta}},
\quad \forall ~ i=1,2,\cdots,d, \label{beta-p} \\
\frac{T_\beta^{-1}(p_{b, \beta_1})}{T_\beta^{-1}(p_{b, \beta_2})} 
\Big(=\frac{\beta_1}{\beta_2}\Big)
& \equiv \frac{\partial_{x_i} \ln p_{b, \beta_2}}{\partial_{x_i} \ln p_{b, \beta_1}},
\quad \forall ~ i=1,2,\cdots,d,  \label{sta-beta}
\end{align}   
\end{tm}

\begin{proof} 
The representation \eqref{lan-d} yields that  
 \begin{align*} 
U(x) = \frac\beta2 \ln p(x) + C.
\end{align*} 
from which we derive \eqref{beta-p} and \eqref{sta-beta}. 
\end{proof}

\begin{rk}
In the case $m \ge d>1$, if $\sqrt \beta$ in the Langevin equation \eqref{eq-lan} is replaced by a constant matrix $\sigma=(\sigma^{ij})_{d \times m}$, then the drift and diffusion inversion problem becomes more complicated.
We point out that the case $m>d>1$ for Eq. \eqref{sode} with $b=\nabla U$ and $\sigma=(\beta {\rm Id}_d ~ {\bf 0}_{d \times (m-d)})$ reduces to Eq. \eqref{eq-lan} which was analyzed in Theorems \ref{tm-lan-b} and \ref{tm-lan-s}.
The general case remains unknown.
\end{rk}

For the high-dimensional Eq. \eqref{sode} driven by additive noise but without the gradient drift, the Fokker-Planck equation \eqref{eq-Fokker-Planck} with $D^{ij}=\delta_{ij} \beta/2$, $i,j=1, 2, \cdots, d$, reduces to
\begin{align} \label{Fokker-Planck-add}
\frac12 \beta \Delta p- \nabla \cdot [b p]=0.  
\end{align} 
Unlike the gradient case, the explicit closed form of the invariant density is not available.
We derive the injectivity of $T_\beta$ (and a different representation from \eqref{beta-p} for its inverse) in the following result, applying the properties of harmonic functions; we note that $T_b$ is generally not injective in high dimensions by Theorem \ref{tm-b-i}.

\begin{tm} \label{tm-s+}
Assume that $(b_k, \beta_k) \in \AAA$ with $\beta_k>0$, $k=1,2$.
Then $\pi_{b, \beta_1}=\pi_{b, beta_2}$ if and only if $\beta_1=\beta_2$, i.e., $T_\beta$ is bijective.
Moreover, $T_\beta^{-1}$ satisfies
\begin{align} 
T_\beta^{-1}(p_{b, \beta})
(=\beta) & \equiv  \frac{\nabla \cdot [b p]}{\Delta p},
\quad \forall ~ i=1,2,\cdots,d, \label{beta-p+} \\
\frac{T_\beta^{-1}(p_{b, \beta_1})}{T_\beta^{-1}(p_{b, \beta_2})} 
\Big(=\frac{\beta_1}{\beta_2} \Big)
& \equiv \frac{\nabla \cdot [b p_1]}{\Delta p_1}/\frac{\nabla \cdot [b p_2]}{\Delta p_2},
\quad \forall ~ i=1,2,\cdots,d.  \label{sta-beta+}
\end{align}   
\end{tm}

\begin{proof}
Let $p_{b, \beta_1}=p_{b, \beta_2}$ with two constants $\beta_1$ and $\beta_2$.
Then Eq. \eqref{Fokker-Planck-add}  yields that $\beta_1-\beta_2$ satisfies 
\begin{align*}
(\beta_1-\beta_2) \Delta p=0.  
\end{align*} 
Suppose that $\beta_1 \neq \beta_2$, then $\Delta p=0$ with $p>0$ satisfying $\int_{\rr^d} p(x) \dd x=1$.
Consequently, $p$ is a positive harmonic function in $\rr^d$ and thus constant.
Indeed, due to the average property of any harmonic function, we have 
 \begin{align*}
p(x) =\frac1{B_R(x)} \int_{B_R(x)} p(y) \dd y, \quad \forall~ x \in \rr^d,
\end{align*} 
where $B_R(\cdot)$ is a ball with radius $R>0$ centered at $x \in \rr^d$.
As $p$ is positive,
\begin{align*}
p(x) \le \frac1{B_R(x)} \int_{\rr^d} p(y) \dd y=\frac1{B_R(x)}, \quad \forall~ x \in \rr^d,
\end{align*} 
Now taking $R \to \infty$, we have 
\begin{align*}
p(x) \le \lim_{R \to \infty} \frac1{B_R(x)}=0, \quad \forall~ x \in \rr^d,
\end{align*} 
from which we obtain $p(x)=0$, which is a contradiction of $\int_{\rr^d} p(x) \dd x=1$.
This shows that $\beta_1= \beta_2$ and the injectivity of $T_\beta$.

The same analysis, in combination with Eq. \eqref{eq-Fokker-Planck}, implies \eqref{beta-p+} and \eqref{sta-beta+}.
\end{proof}

We end this section by giving an example about the coefficients of Eq. \eqref{sode} such that all the counterexamples given in Section \ref{sec3} hold.
 
\begin{ex}  \label{ex1}
Assume that $b$ and $\sigma$ in Eq. \eqref{sode} satisfy the following coupled conditions of monotone, coercive, and polynomial growth for some positive constants $L_i$, $i=0,1,2,3,4$ and $q$, and for all $x,y \in \rr^d$:
\begin{align} 
& 2 \<b(x)-b(y), x-y\>+ \|\sigma(x)-\sigma(y)\|^2 
\le L_0|x-y|^2,  \label{mon-} \\
& 2 \<b(x), x\> + \|\sigma(x)\|^2 \le L_1 - L_2 |x|^2,  \label{coe-} \\
& |b(x)| \le L_3+L_4|x|^q, \label{pol-} \\
& D(x) =\sigma(x) \sigma(x)^\top \in \rr^{d \times d}~\text{is positive definite for any}~x \in \rr^d.  \label{non} 
\end{align}   
The above conditions, in particular, include polynomial drift (with a negative leading coefficient) and diffusion functions with $2a_{2k+1}+c_{k+1}^2<0$: 
 \begin{align*}
b(x)=\sum_{i=0}^{2k+1} a_i x^i, \quad
\sigma(x)=\sum_{j=0}^{k+1} c_j x^j,
\quad x \in \rr,
 \end{align*} 
 where $k \in \nn_+$, $a_i, c_j \in \rr$, $i=0,1,\cdots,2k+1$, $j=0,1,\cdots,k+1$.   
It is known that under these conditions, Eq. \eqref{sode} is uniquely ergodic (see, e.g., \cite{LL25}).
\end{ex}

\section{Drift and diffusion Inversion for SPDEs driven by additive noise}
\label{sec4}

In this section, we apply the method developed in Section \ref{sec3} to an infinite-dimensional SDE or its equivalent SPDE form.
The infinite-dimensional case is considerably more complicated than the previous finite-dimensional case, partly due to the absence of exact analogs of Lebesgue measure.

For simplicity, we consider the following second-order parabolic SPDE driven by additive noise under the homogeneous Dirichlet boundary condition on the separable Hilbert state space $H=\{u \in L^2(0,1):~ u(0)=u(1)=0\}$:
\begin{align}\label{spde}
\begin{split}
    \dd X(t,\xi) =\Delta X(t,\xi)+ U'(X(t,\xi)) \dd t + \sqrt \beta \dd W(t,\xi), 
   & \quad (t,\xi) \in \rr_+ \times (0,1); \\
    \quad X(0,\xi)=X_0(\xi), & \quad \xi \in (0,1).
    \end{split}
\end{align}
Here, $X_0 \in H$, $\beta>0$ is a constant, and $\{W(t, \cdot)=\sum_{i=1}^\infty B_t^i e^i(\cdot),~  t\ge 0\}$ is a cylindrical Wiener process on $H$ with $\{e^i\}_{i\ge 1}$ and $\{B_\cdot^i\}_{i\ge 1}$ being an orthonormal basis of $H$ and a sequence of independent one-dimensional Brownian motions, respectively; the temporal derivative of $W$ in the sense of distribution is the so-called (space-time) white noise.  
 
Though there is no Lebesgue measure in $H$, we choose the reference measure to be the centered Gaussian measure $\mu_0$ on $H$ with covariance operator $(-\Delta)^{-1}$ (which is of trace-class). 
Indeed, it is clear that $-\Delta$ has discrete spectrum with positive eigenvalues $\{\lambda_i=(k \pi)^2\}_{i\ge 1}$ listed in the increasing order counting multiplicities satisfying ${\rm Tr (-\Delta)}:=\sum_{i \ge 1} \lambda_i=1/6<\infty$.

Denote by $\AAA$ agian the set of assumptions on $(b, \beta)$ such that Eq. \eqref{sode} has a unique mild (which is probabilistically strong) solution with a unique ergodic measure $\pi_{b, \beta} \in \PP(H):=\{\text{probabily measures on}~H\}$ having a positive density $p_{b, \beta}$ with respect to $\mu_0$ on $H$.
We note that $p_{b, \beta}$ is $\mu_0${\rm -a.e.} uniquely determined by $\pi_{b, \beta}$.
It is clear $\AAA$ is nonempty as this assumption is valid under broad conditions that include all the examples given in Section \ref{sec4}; see, e.g., \cite{Liu25}. 

For each pair of $(b, \beta) \in \AAA$, we obtain a solution $\pi_{b, D} \in \PP(H)$ with density $p_{b, \beta} \in \DD(H):=\{\text{density functions on}~H~ \text{with respect to}~\mu_0\}$ and we can define the measurement operator 
\begin{align}  \label{t-s+}
\begin{split}
T_\beta: \AAA_\beta =\{b:~(b, \beta) \in \AAA\} & \longrightarrow T_{\beta}(\AAA_\beta) \subset \DD(H) \\
b & \longmapsto p_{b, \beta}=T_\beta(b, \beta),
\end{split}
\end{align} 
for a fixed $\beta$, and  
\begin{align}  \label{t-b+}
\begin{split}
T_D: \AAA_D =\{D:~(b, \beta) \in \AAA\} & \longrightarrow T_D(\AAA_D) \subset \DD(H) \\
D & \longmapsto p_{b, \beta}=T_D(b, \beta),
\end{split}
\end{align}  
for a fixed $b$. 
By the definitions of $\AAA_\beta$ and $\AAA_D$ in \eqref{t-s+} and \eqref{t-b+}, respectively, it is clear that $T_\beta: \AAA_\beta \to T_\beta(\AAA_\beta)$ and $T_D: \AAA_D \to T_D(\AAA_\beta)$ are surjective and usually nonlinear.
Our main aim is to determine whether $T_\beta: \AAA_\beta \to T_\beta(\AAA_\beta)$ and $T_D: \AAA_D \to T_D(\AAA_\beta)$ are injective, and, if they are injective (and thus bijective), determine the representation of their inverses, denoted by $T_\beta^{-1}$ and $T_D^{-1}$, as in the finite-dimensional cases studied in Theorems \ref{tm-lan-b} and \ref{tm-s+}.

It is not difficult to show that the Gibbs measure  
\begin{align} \label{lan-spde}
\mu(\dd x):= Z_U^{-1} e^{\frac{2 U(x)}\beta}\mu_0(\dd x), \quad x \in H,
\end{align}
is the unique invariant measure, in the sense of \eqref{df-im} and \eqref{df-erg}, where $H$ there and \eqref{df-pt} is replaced by $H$, of Eq. \eqref{spde} in $H$ % and $P_t$ is symmetric in $L^2(\mu)$.  
provided 
\begin{align} \label{con-spde}
Z_U:= \mu_0(e^{\frac{2 U(x)}\beta})<\infty. 
\end{align}

The following result is a generalization of the finite-dimensional cases studied in Theorems \ref{tm-lan-b} and \ref{tm-s+}, which shows the bijectivity of $T_b$ and $T_\beta$ and the representations of their inverses.

\begin{tm} \label{tm-spde}
Assume that $(b_k, \beta_k) \in \AAA$ with constants $\beta_k>0$ satisfy \eqref{con-spde} for any premitive $U_k \in \CC^1(\rr; \rr)$ of $b_k$, $k=1,2$.
\begin{enumerate}
\item
$T_b$ is bijective and its inverse $T_b^{-1}$ satisfies $\mu_0${\rm -a.e.}
\begin{align}  
T_b^{-1}(p_{b, \beta}) (=b) 
& =\frac\beta2 \nabla \ln p_{b, \beta}(x),  \label{b-p-spde} \\
T_b^{-1}(p_{b_1, \beta})-T_b^{-1}(p_{b_2, \beta}) (=b_1-b_2)
& =\frac\beta2 [\nabla \ln p_{b_1, \beta}(x)-\nabla \ln p_{b_2, \beta}(x)].  \label{sta-b-spde}
\end{align}

\item
$T_\beta$ is bijective and its inverse $T_\beta^{-1}$ satisfies $\mu_0$-{\rm a.e.}
\begin{align}  \label{beta-p-spde}
T_\beta^{-1}(p_{b, \beta_1}) (=\beta) & =\frac{2 \<e_k, b(x)\>}{D_{e_k} \ln p(x)}, 
\quad \forall ~ k \ge 1, \\ 
\frac{T_\beta^{-1}(p_{b, \beta_1})}{T_\beta^{-1}(p_{b, \beta_2})} (=\frac{\beta_1}{\beta_2})
 & \equiv \frac{D_{e_k} \ln p_{b, \beta_2}}{D_{e_k} \ln p_{b, \beta_1}}\quad \forall ~ k \ge 1, \label{sta-beta-spde}
\end{align} 
where $D_{e_k}$ is the directional derivative with respect to $e_k$
\end{enumerate}  
\end{tm}

\begin{proof} 
It follows from the representation \eqref{lan-spde} that 
\begin{align*}
\ln p(x)= \frac{2 U(x)}{\beta} +C,  
\end{align*}  
from which we derive \eqref{b-p-spde} and \eqref{sta-b-spde}. 

Taking the directional derivative $D_{e_k}$ with respect to $e_k$ in the above identity, we have 
\begin{align*}
D_{e_k} \ln p(x(\xi))= \frac{2 D_{e_k} U(x(\xi))}{\beta}= \frac{2 \<e_k, b(x)\>}{\beta}.
\end{align*}  
and conclude \eqref{beta-p-spde} and \eqref{sta-beta-spde}.
\end{proof}

We end this section by giving an example about the coefficients of Eq. \eqref{spde} such
that all the assumptions given in Section \ref{sec4} hold.

\begin{ex}  \label{ex2}
Assume that there exist constants $K_1, K_2 \in \rr$, $K_3>\lambda_1$, $K_4, K_5>0$, and $q \ge 1$ such that for any $\xi, \eta \in \rr$,
\begin{align}  
[b(\xi)- b(\eta)](\xi-\eta)  & \le (K_1 + \lambda_1)|\xi-\eta|^2, \label{mon}  \\
b(\xi) \xi & \le K_2 - (K_3 - \lambda_1) |\xi|^2,  \label{coe} \\
|b(\xi)| & \le K_4 + K_5 |\xi|^q. \label{grow}
\end{align}
Under these conditions, the assumption \eqref{con-spde} holds, and Eq. \eqref{spde} has a unique mild solution which is uniquely ergodic in $H$ with respect to the invariant measure given in \eqref{lan-spde}; see, e.g., \cite{Liu25, Wan23} for ergodic analysis.
\end{ex}

 \bibliographystyle{plain}
  \bibliography{bib.bib}

\end{document}